\documentclass[12pt]{amsart}
\usepackage[margin=1.25in]{geometry}
\usepackage[utf8]{inputenc}
\usepackage{xcolor}
\usepackage{hyperref}
\usepackage{multirow}
\usepackage{booktabs} 
\usepackage{array}    

\usepackage{amssymb}

\makeatletter
\newcommand\reallytiny{\@setfontsize\reallytiny{5}{6}}
\makeatother
\newcommand{\unit}{\text{\textbf{1}}}

\renewcommand{\Vec}{\text{Vec}}

\newcommand{\cC}{{\mathcal C}}

\newcommand{\FPdim}{{\operatorname{FPdim}}}

\newcommand{\Z}{{\mathbb Z}}

\newcommand{\id}{\operatorname{id}}

\newcommand{\cB}{\mathcal{B}}

\newcommand{\Id}{\operatorname{Id}}

\theoremstyle{plain}

\numberwithin{equation}{section}

\newtheorem{theorem}{Theorem}[section]

\newtheorem{corollary}[theorem]{Corollary}
\newtheorem{proposition}[theorem]{Proposition}

\newcommand{\xdownarrow}[1]{%
  {\left\downarrow\vbox to #1{}\right.\kern-\nulldelimiterspace}
}

\theoremstyle{definition}
\newtheorem{definition}[theorem]{Definition}
\newtheorem{example}[theorem]{Example}

\theoremstyle{remark}

\newtheorem{remark}[theorem]{Remark}

\author[C. Galindo]{C\'esar Galindo}
\address{ Departamento de Matem\'aticas, Universidad e los Andes, Bogot\'a, Colombia}
\email{cn.galindo1116@uniandes.edu.co}
\author[J. Plavnik]{Julia Plavnik}
\address{Department of Mathematics, Indiana University}
\email{jplavnik@iu.edu}
\author[E. Rowell]{Eric C. Rowell}
\address{ Department of Mathematics, Texas A\&M University, College Station, TX, USA}
\email{rowell@tamu.edu}
\begin{document}

\title[Integral non-group-theoretical  modular categories of dimension $p^2q^2$]{Integral non-group-theoretical modular categories of dimension $p^2q^2$}

\thanks{C.G. was partially supported by Grant INV-2023-162-2830 from the School of Science of Universidad de los Andes. J.P. was partially supported by US NSF grant DMS-2146392 and by  Simons Foundation Award 889000 as part of the Simons Collaboration on Global Categorical Symmetries. E.C.R. was partially supported by US NSF grant DMS-2205962.}

\begin{abstract}
We construct all integral non-group-theoretical modular categories of dimension \(p^2q^2\), where \(p\) and \(q\) are distinct prime numbers, establishing that a necessary and sufficient condition for their existence is that \(p \mid q+1\), and their rank is \(p^2 + \frac{q^2 - 1}{p}\).
\end{abstract}


\date{\today}
\maketitle

\section{Introduction}
While integral non-group-theoretical modular categories of dimension $4q^2$ with $q$ an odd prime were constructed in \cite{GNN}, a sign error in \cite{AIM2012} led to the mistaken conclusion that for odd primes $p,q$ there were no non-group-theoretical categories of dimension $p^2q^2$.  Recently this error was pointed out to us by Palcoux, with a potential rank 17 counterexample for $p=3$, $q=5$ described in  \cite{Palcouxetc}.  

We correct this oversight here by explicitly constructing all non-group theoretical modular categories of dimension $p^2q^2$ for $q$ an odd prime with $p$ a prime dividing $q+1$.  The case where $p$ divides $q-1$ was already handled in \cite{AIM2012}, but here we deal with all cases simultaneously. Therefore, our construction covers all integral non-group-theoretical modular categories of dimension $p^2q^2$ including dimension \(4q^2\) for any odd \(q\), and \(\mathcal{C}(\mathfrak{sl}_2, q, 6)\) and their zestings, \cite{zesting}.

The paper is organized as follows. In Section \ref{sec: preliminaries}, we recall some basic definitions necessary for the general construction of modular categories associated with faithful actions of cyclic groups on metric groups. In Section \ref{sec: definition of modular cat}, for each quadratic extension of a finite field, we construct an anisotropic metric group whose  orthogonal group is a dihedral group. With this we define $\left(\operatorname{Vec}_{(\mathbb{F}_{q^2}, N)}^{p, \alpha}\right)^{\mathbb{Z}/p\mathbb{Z}}$ the family of integral non-group-theoretical modular categories of dimension $p^2q^2$. In Section \ref{sec: propiedades}, we parameterize the simple objects of $\left(\operatorname{Vec}_{(\mathbb{F}_{q^2}, N)}^{p, \alpha}\right)^{\mathbb{Z}/p\mathbb{Z}}$ and prove that it is non-group-theoretical. Finally, in Section \ref{sec: group-theoretical}, we show that group-theoretical modular categories of dimension $p^2q^2$ are either the representations of a twisted Drinfeld double of a non-abelian group of order $pq$ or pointed, providing a complete description of them.

\section{Preliminaries}\label{sec: preliminaries}

By a \emph{fusion category}, we mean a $\mathbb{C}$-linear rigid semisimple tensor category with finitely many isomorphism classes of simple objects and simple unit object $\mathbf{1}$. For basic definitions, including those of braided fusion categories, modular categories, and the 2-category of (bi)module categories over a fusion category, along with their properties and examples, we refer the reader to \cite{EGNO}.

An object $X$ in a fusion category is called \emph{invertible} if $X \otimes X^* \cong X^* \otimes X \cong \mathbf{1}$. A fusion category $\mathcal{B}$ is called \emph{pointed} if every simple object is invertible. Up to tensor equivalence, every pointed fusion category is equivalent to $\operatorname{Vec}_G^{\omega}$, the category of finite-dimensional $G$-graded vector spaces, where $G$ is a finite group and the associativity constraint is given by a $3$-cocycle $\omega \in Z^3(G, \mathbb{K})$. A fusion category $\mathcal{C}$ is called \emph{group-theoretical} if it is Morita equivalent to a pointed fusion category, this means if there exists a $\mathcal{C}$-module category $\mathcal{M}$ such that $\operatorname{End}_{\mathcal{C}}(\mathcal{M})$ is a pointed fusion category, see \cite{ENO}.

Given a braided fusion category $\mathcal{D}$ and a braided inclusion $\operatorname{Rep}(G) \hookrightarrow \mathcal{D}$, we denote by $\mathcal{D}_G$ the braided $G$-crossed category constructed through \emph{de-equivariantization}, as described in \cite{DGNO}. Additionally, $G$ acts on $\mathcal{B} := (\mathcal{D}_G)_e$—the trivial component of the associated $G$-crossed braided  category—via braided tensor autoequivalences. This action establishes a group homomorphism $G \to \operatorname{Aut}_{\otimes}^{\operatorname{br}}(\mathcal{B})$.

Conversely, if we start with a non-degenerate braided fusion category $\mathcal{B}$ and a group homomorphism $G \to \operatorname{Aut}_{\otimes}^{\operatorname{br}}(\mathcal{B})$, the method known as \emph{gauging} (see \cite{CGPW}, \cite{BBCW}) enables us to build a braided $G$-crossed category called $\mathcal{B}^{\times, G}$. Following this, a non-degenerate braided category $\mathcal{D}$, obtained through the $G$-equivariantization of $\mathcal{B}^{\times, G}$ contains $\operatorname{Rep}(G)$.
 The categories $\mathcal{B}^{\times, G}$ are classified by pairs $(M, \alpha)$, where $M$ and $\alpha$ are elements of torsors over $H^2(G,\operatorname{Inv}(\mathcal{B}))$ and $H^3(G, \mathbb{C}^\times)$, respectively. This classification relies on the vanishing of some cohomological obstructions, $o_3(\rho) \in H^3(G,\operatorname{Inv}(\mathcal{B}))$ and $o_4(\rho, M) \in H^4(G, \mathbb{C}^\times)$ \cite{ENO3}.

When $G = \mathbb{Z}/n\mathbb{Z}$ with $n$ coprime to the size of $\operatorname{Inv}(\mathcal{B})$, the corresponding non-degenerate modular categories are uniquely characterized by $H^3(\mathbb{Z}/n\mathbb{Z}, \mathbb{C}^\times) \cong \mathbb{Z}/n\mathbb{Z}$. Their existence is guaranteed as $H^n(\mathbb{Z}/n\mathbb{Z},\operatorname{Inv}(\mathcal{C})) = 0$ and $H^4(\mathbb{Z}/n\mathbb{Z},\mathbb{C}^\times) = 0$.

Based on the preceding discussion and in order to establish notation, we introduce the following definition, which precisely corresponds to the type of modular categories we are interested in constructing.

\begin{definition}\label{def: construction}
Let $\mathcal{B}$ be a non-degenerate braided fusion category. For a cyclic subgroup $\langle T \rangle\subseteq \operatorname{Aut}_{\otimes}^{\operatorname{br}}(\mathcal{B})$ of order $n$, coprime to $|\operatorname{Inv}(\mathcal{B})|$, we denote by $\mathcal{B}^{(\langle T \rangle,\alpha)}$ the corresponding braided $\mathbb{Z}/n\mathbb{Z}$-crossed extension of $\cB$, where $\alpha$ is an element of $H^3(\mathbb{Z}/n\mathbb{Z},\mathbb C^\times)\cong \mathbb{Z}/n\mathbb{Z}$. Additionally, we denote by
$(\mathcal{B}^{(\langle T \rangle,\alpha)})^{\mathbb{Z}/n\mathbb{Z}}$  the associated non-degenerate braided fusion category.
\end{definition}

Recall a metric group is a pair $(A, t)$, where $A$ is a finite abelian group and $t: A \to \mathbb{C}^*$ is a non-degenerate quadratic form. This means that the map defined by $(a, b) \mapsto \frac{t(a + b)}{t(a)t(b)}$ is a non-degenerate bicharacter and $t(a) = t(-a)$. A pointed modular category $\mathcal{B}$ gives rise to a metric group by taking $A = \operatorname{Inv}(\mathcal{B})$, the set of isomorphism classes of invertible objects, and $t: A \to \mathbb{C}^*$ given by $c_{a, a} = t(a) \operatorname{id}_{a \otimes a}$. Conversely, every metric group has an associated pointed modular category (see \cite{DGNO} for details).

A key example of a metric group, and consequently of a pointed modular category, is constructed as follows: let \(V\) be a finite-dimensional vector space over a finite field \(\mathbb{F}\) of characteristic \(p\), and let \(Q: V \to \mathbb{F}\) be an ordinary non-degenerate quadratic form. We then define the metric group \((V, t)\) where \(t(v) = e^{\frac{2 \pi i N(Q(v))}{p}}\), with \(N:\mathbf{F}_p\to \mathbf{F}_p\) being the field norm. This construction provides an associated metric group and, consequently, a pointed modular category.

The group of braided tensor autoequivalences of a pointed modular category with associated metric group \((A, t)\) is naturally isomorphic to \(\operatorname{Aut}(A, t)\), the group of automorphisms of \(A\) that fix \(t\). In the case of a non-degenerate quadratic \(\mathbb{F}\)-linear space \((V,Q)\), it holds that \(O(V,Q) \subset \operatorname{Aut}(V,t)\). Therefore, given a cyclic subgroup of \(O(V,Q)\) of order relatively prime to \(p\), we can construct a modular category via gauging, as in Definition \ref{def: construction}. This method constructs the family of non-group-theoretical modular categories of dimension \(p^2q^2\).

\section{Definition of Modular Categories $\left(\operatorname{Vec}_{(\mathbb{F}_{q^2}, N)}^{p, \alpha}\right)^{\mathbb{Z}/p\mathbb Z}$}\label{sec: definition of modular cat}

Let \(q\) be a prime and let \(\mathbf{F}_{q} \subset \mathbf{F}_{q^{2}}\) denote a finite Galois extension, where \(\mathbf{F}_{q}\) and \(\mathbf{F}_{q^{2}}\) are fields with \(q\) and \(q^{2}\) elements, respectively. We denote the generator of the Galois group by $\sigma: \mathbf{F}_{q^{2}} \to \mathbf{F}_{q^{2}}$, so $\sigma (v) = v^q$, for $v\in \mathbf{F}_{q^{2}}$. The norm \(N: \mathbf{F}_{q^{n+1}} \to \mathbf{F}_q\) is defined by 
\[N(v) = v\sigma(v) = v^{q+1},\]
which establishes an anisotropic plane, that is a 2-dimensional \(\mathbf{F}_{q}\)-vector space equipped with the quadratic form \(N\) satisfying \(N(v) = 0\) if and only if \(v = 0\).
We denote by $\mathbf{F}_{q^{2}}^* = \operatorname{Hom}_{\mathbf{F}_{q}}(\mathbf{F}_{q^{2}}, \mathbf{F}_{q})$. This norm induces a group epimorphism $N: \mathbf{F}_{q^{2}}^{*} \to \mathbf{F}_{q}^{*}$, with $\ker(N)$ being a cyclic group of order $q^2$.

The assignment
\begin{align*}
    \rho: \ker(N) &\to SO(\mathbf{F}_{q^{2}}, N), \\
    c &\mapsto [\rho_c: v \mapsto cv],
\end{align*}
defines a group isomorphism. Now, since $\sigma \in O(\mathbf{F}_{q^{2}}, N)$ with $\det(\sigma) = -1$ (for $q=2$ Dickson's pseudodetermiant non trivial) and $\sigma^2=\Id$, it follows that $O(\mathbf{F}_{q^{2}}, N) = SO(\mathbf{F}_{q^{2}}, N) \rtimes \langle \sigma \rangle$. Consequently, $O(\mathbf{F}_{q^{2}}, N)$ is isomorphic to the dihedral group of order $2(q + 1)$.

\begin{definition}\label{def: modular}
Let $p$ and $q$ be primes with $p \mid q + 1$. Given $1 \neq c \in \mathbb{F}_{q^2}$ such that $N(c) = 1$ and $c^p = 1$, we define $\operatorname{Vec}_{(\mathbb{F}_{q^2}, N)}^{p, \alpha}$, where $\alpha \in H^3(\mathbb{Z}/p\mathbb Z,\mathbb C^*)\cong \mathbb{Z}/p\mathbb Z$, as the associated braided $\mathbb{Z}/p\mathbb Z$-crossed modular category. We denote by $\left(\operatorname{Vec}_{(\mathbb{F}_{q^2}, N)}^{p, \alpha}\right)^{\mathbb{Z}/p\mathbb Z}$ the corresponding modular category obtained by $\mathbb{Z}/p\mathbb{Z}$-equivariantization.
\end{definition}

\begin{remark}\label{rmk: sobre la cartegoria modular}
\begin{enumerate}
\item As we will see in Proposition~\ref{prop: fusion rules}, the fusion category $\Vec_{(\mathbb{F}_{q^2}, N)}^{p, \alpha}$ is integral. Consequently, its equivariantization, the modular category $\left(\Vec_{(\mathbb{F}_{q^2}, N)}^{p, \alpha}\right)^{\mathbb{Z}/p\mathbb{Z}}$, is also integral. Given that every integral fusion category has a unique spherical structure with quantum dimensions matching the Frobenius-Perron dimensions, this category is indeed modular.

\item Since the orthogonal group of $(\mathbf{F}_{q^2}, N)$ is a dihedral group, every odd cyclic subgroup is completely determined by its order and for order two they correspond to reflections. Hence, $\left(\Vec_{(\mathbb{F}_{q^2}, N)}^{p, \alpha}\right)^{\mathbb{Z}/p\mathbb{Z}}$ does not depend on the choice of $c$ in the case of $p$ odd and for $p=2$ the category $\Vec_{(\mathbb{F}_{q^2}, N)}^{p, \alpha}$ corresponds to a Tambara-Yamagami category.

\item The modular category \(\left(\operatorname{Vec}_{(\mathbb{F}_{q^2}, N)}^{p, \alpha}\right)^{\mathbb{Z}/p\mathbb{Z}}\) is \(\mathbb{Z}/p\mathbb{Z}\)-graded, with  trivial component \(\operatorname{Vec}_{(\mathbf{F}_{q^2}, N)}^{\mathbb{Z}/q\mathbb{Z}}\). The trivial component can be realized simply as the category of representations of the semi-direct product \(\mathbf{F}_{q^2} \rtimes_c \mathbb{Z}/p\mathbb{Z}\), where \(\mathbf{F}_{q^2}\) is considered only as an abelian group and thus \(\mathbf{F}_{q^2} \cong (\mathbb{Z}/q\mathbb{Z})^2\). 
\item The modular category \(\left(\operatorname{Vec}_{(\mathbb{F}_{q^2}, N)}^{p, \alpha}\right)^{\mathbb{Z}/p\mathbb{Z}}\) 
is a minimal modular extension of \(\operatorname{Vec}_{(\mathbf{F}_{q^2}, N)}^{\mathbb{Z}/q\mathbb{Z}}\), and  this minimal modular extensions are unique up to 
twisting by elements in $H^3(\mathbb{Z}/p\mathbb{Z},\mathbb C^\times)\cong \mathbb{Z}/p\mathbb{Z}$, see \cite{LKW}.
\end{enumerate}
\end{remark}

\begin{example}\label{Examples}
\begin{enumerate}
    \item If $q=2$, then $\mathbb{F}_{4}=\mathbb{F}_{2}[\alpha]$, with $\alpha^2+\alpha+1=0$. In this case, $\mathbb{F}_{4}=\langle 1, \alpha \rangle$ as an abelian group and $t(1)=t(\alpha)=-1$. Taking $c=\alpha$, we have that $\rho(c)$ has order three. Consequently, $\operatorname{Vec}_{(\mathbb{F}_{4}, N)}^{\mathbb{Z}/3\mathbb{Z}}$ is equivalent to the representation category of $\mathbb{F}_{4} \rtimes \mathbb{Z}/3\mathbb{Z} \cong \mathbb{S}_3$ with a non-symmetric braiding. Hence, $\left(\operatorname{Vec}_{(\mathbb{F}_{4}, N)}^{3, \alpha}\right)^{\mathbb{Z}/3\mathbb{Z}}$ is braided equivalent to $\mathcal{C}(\mathfrak{sl}_2, q, 6)$ or one of their zesting, \cite{zesting}.
    \item If $q$ is an odd prime and $p=2$, the braided $\mathbb{Z}/2\mathbb{Z}$-crossed category $\operatorname{Vec}_{(\mathbb{F}_{q^2}, N)}^{2, \alpha}$ is a Tambara-Yamagami category. The associated modular category corresponds to the elliptic case of \cite[Example 5.3]{GNN}.
\end{enumerate}

\end{example}

\section{Properties of $\operatorname{Vec}_{(\mathbb{F}_{q^2}, N)}^{p, \alpha}$}\label{sec: propiedades}

From now on, we will assume that $p$ and $q$ are odd primes, since the even cases were already discussed in Example \ref{Examples} and correspond to well-known examples of integral non-group-theoretical modular categories.

We denote by $\mathbf{F}_{q^2}^* = \operatorname{Hom}_{\mathbf{F}_q}(\mathbf{F}_{q^2}, \mathbf{F}_q)$ and $O(\mathbf{F}_{q^2} \oplus \mathbf{F}_{q^2}^*, Q)$ the \emph{split orthogonal group}, where $Q: \mathbf{F}_{q^2} \oplus \mathbf{F}_{q^2}^* \to \mathbf{F}_q$ is given by $Q(v, \alpha) = \alpha(v)$.

Let $B(v, w) = \frac{1}{2}[N(v + w) - N(v) - N(w)]$ be the associated bilinear form of $(\mathbf{F}_{q^2}, N)$. Since $B$ is non-degenerate, it defines a map $\widehat{(-)}: \mathbf{F}_{q^2} \to \mathbf{F}_{q^2}^*$, with $\widehat{v}(w) = B(v, w)$ for all $w \in \mathbf{F}_{q^2}$.

Then, we have orthogonal injections
\begin{align*}
  (\mathbf{F}_{q^2}, N) &\to (\mathbf{F}_{q^2} \oplus \mathbf{F}_{q^2}^*, Q), & v &\mapsto (v, \widehat{v}), \\
  (\mathbf{F}_{q^2}, -N) &\to (\mathbf{F}_{q^2} \oplus \mathbf{F}_{q^2}^*, Q), & v &\mapsto (v, -\widehat{v}).
\end{align*}

There exists a unique injective group homomorphism \(\alpha: O(\mathbf{F}_{q^2}, N) \to O(\mathbf{F}_{q^2} \oplus \mathbf{F}_{q^2}^*, Q)\), characterized by \(\alpha_g(v, \widehat{v}) = (v, \widehat{v})\) and \(\alpha_g(v, -\widehat{v}) = (g(v), -\widehat{g(v)})\), see \cite{davydov2021braided}. Then
\begin{align*}
    \alpha_g(v, \widehat{w}) &= \alpha_g\left(\frac{1}{2}[ (v + w, \widehat{v} + \widehat{w}) + (v - w, -(\widehat{v} - \widehat{w}))]\right) \\
    &= \frac{1}{2}[(v + w, \widehat{v} + \widehat{w}) + (g(v) - g(w), -(\widehat{g(v)} - \widehat{g(w)}))] \\
    &= \frac{1}{2}[( (\mathrm{Id} + g)(v), \widehat{(\mathrm{Id} - g)(v)}) + ((\mathrm{Id} - g)(w), \widehat{(\mathrm{Id} + g)(w)})].
\end{align*}

Hence, we can write
\begin{align}\label{eq:definition_matrix}
\alpha_g = \begin{pmatrix}
    \alpha & \beta \\
    \gamma & \delta
\end{pmatrix} \in O(\mathbf{F}_{q^2} \oplus \mathbf{F}_{q^2}^*, Q),
\end{align}
where
\begin{align*}
    \alpha = \frac{1}{2}(\mathrm{Id} + g), &\quad& \beta(\widehat{w}) = \frac{1}{2}(\mathrm{Id} - g)(w), \\
    \gamma(v) = \frac{1}{2}\widehat{\mathrm{Id} - g(v)}, &\quad& \delta(\widehat{w}) = \frac{1}{2}\widehat{\mathrm{Id} + g(w)},
\end{align*}
and \(\alpha: \mathbf{F}_{q^2} \to \mathbf{F}_{q^2}\), \(\beta: \mathbf{F}_{q^2}^* \to \mathbf{F}_{q^2}\), \(\gamma: \mathbf{F}_{q^2} \to \mathbf{F}_{q^2}^*\), and \(\delta: \mathbf{F}_{q^2}^* \to \mathbf{F}_{q^2}^*\).

\begin{proposition}\label{prop: fusion rules}
The simple objects in the fusion category $\mathcal{B} := \operatorname{Vec}_{(\mathbb{F}_{q^2}, N)}^{p, \alpha}$ are given by $\mathbf{F}_{q^2}$, which correspond to the group of (isomorphism classes of) invertibles in $
\mathcal B$, and non-invertibles $X_1, \ldots, X_{p-1}$, which are q-dimensional. The fusion rules are given by
\begin{align*}
 a \otimes b &= a + b, & a \otimes X_i = X_i \otimes a &= X_i, & X_i^* &= X_{p-i},
\end{align*}
and
\[
X_i \otimes X_j = 
\begin{cases} 
qX_{i+j} & \text{if } i + j \neq p, \\
\sum_{a \in \mathbf{F}_{q^2}} a & \text{if } i + j = p,
\end{cases}
\]

for all $a, b \in \mathbf{F}_{q^2}$.
\end{proposition}

\begin{proof}
Since $\mathcal B$ has dimension $pq^2$, by \cite[Proposition 3.1]{jordan-larson}, we only need to verify that \(\beta = \frac{1}{2}(\mathrm{Id} - g)\), as defined in \eqref{eq:definition_matrix}, is invertible. In our case, \((\mathrm{Id} - g)(v) = (1 - c^m)v\), for all \(v \in \mathbf{F}_{q^2}\) and \(0 \leq m < p\). Since \(c\) has order \(p\) and the operator \(\frac{1}{2}(\mathrm{Id} - g)\) is invertible, the criteria are satisfied.
\end{proof}

\begin{proposition}\label{prop:construction}
The modular category \(\left(\operatorname{Vec}_{(\mathbb{F}_{q^2}, N)}^{p, \alpha}\right)^{\mathbb{Z}/p\mathbb{Z}}\) has rank \(p^2 + \frac{q+1}{p}\). The following is a complete list of simple objects and their dimensions:

\begin{itemize}
    \item[(1)] There are exactly $p$ invertible objects \((\unit,\chi)\), indexed by \(\chi \in \widehat{\mathbb{Z}/p\mathbb{Z}}\). The corresponding object to \((\unit,\chi)\) is the unit object \(\unit\) with equivariant structure \(\chi(a)\id_\unit: \unit \to \unit\).
    \item[(2)] There are exactly \(\frac{q^2-1}{p}\) simple objects of dimension $p$, parameterized by the \(\mathbb Z/q\)-orbits of \(\mathbf{F}_{q^2}\). The corresponding object to an orbit \(\mathcal{O}\) is the object \(X_{\mathcal{O}}=\bigoplus_{a\in \mathcal{O}} a\) with equivariant structure \(\id_{X_{\mathcal{O}}}\).
    \item[(3)] There are exactly \((p-1)p\) simple objects of dimension \(q\), parameterized by pairs \((X_i, \chi)\), where \(i \in \mathbb{Z}/p\mathbb{Z}^*\) and \(\chi \in \widehat{\mathbb{Z}/p\mathbb{Z}}\). The corresponding object to the pair \((X_i, \chi)\) is the object \(X_i\) with equivariant structure \(\chi(a)\id_{X_i} : X_i \to X_i\).
\end{itemize}
\end{proposition}
\begin{proof}
Since  $\operatorname{Vec}_{(\mathbb{F}_{q^2}, N)}^{p, \alpha}$ is a braided \(\mathbb{Z}/p\mathbb{Z}\)-crossed category, the \(\mathbb{Z}/p\mathbb{Z}\)-action respects the grading. Therefore, at the level of objects, it acts trivially on \(X_i\), and in the trivial component, the action is given exactly by \(\rho_c: \mathbf{F}_{q^2} \to \mathbf{F}_{q^2}\). Then, the description of the simple objects follows from straightforward computations, see \cite{burciu2013fusion}.
\end{proof}

The following criterion will be useful in determining whether $\operatorname{Vec}_{(\mathbb{F}_{q^2}, N)}^{p, \alpha}$ is non-group-theoretical.

\begin{proposition}
\label{Prop:criterion}
Let $p$ and $q$ be odd primes such that $p \mid q + 1$. Let $A = (\mathbb{Z}/q\mathbb{Z})^2$ and \[M = \begin{pmatrix}
    \alpha & \beta \\
    \gamma & \delta
\end{pmatrix} \in O(A \oplus A^*,Q), \quad (\text{the split orthogonal group})\] a matrix of order $p$, where $\alpha: A \to A$, $\beta: A^* \to A$, $\gamma: A \to A^*$, and $\delta: A^* \to A^*$, and $\beta$ is invertible. Then the $\mathbb{Z}/p\mathbb{Z}$-graded extensions of $\operatorname{Vec}_A$ associated with $M$ are group-theoretical if and only if $\mu_1/\mu_2 \in \mathbf{F}_q$, where $\mu_1, \mu_2 \in \mathbf{F}_{q^2}^*$ are the eigenvalues of $\alpha + \beta \delta \beta^{-1}$.
\end{proposition}

\begin{proof}
By \cite[Theorem 3.2]{jordan-larson}, we have

\[\begin{pmatrix}
    \mathrm{Id} & 0 \\ -\delta \beta^{-1} & \mathrm{Id}
\end{pmatrix} \begin{pmatrix}
    \alpha & \beta \\ \gamma & \delta
\end{pmatrix} \begin{pmatrix}
    \mathrm{Id} & 0 \\ \delta \beta^{-1} & \mathrm{Id}
\end{pmatrix} = \begin{pmatrix}
    \alpha - \beta \delta \beta^{-1} & \beta \\ \beta^{-1*} & 0
\end{pmatrix} \in O(A \oplus A^*).\]

Now consider

\[\begin{pmatrix}
    \mathrm{Id} & 0 \\ 0 & \beta
\end{pmatrix} \begin{pmatrix}
    \alpha - \beta \delta \beta^{-1} & \beta \\ \beta^{-1*} & 0
\end{pmatrix} \begin{pmatrix}
    \mathrm{Id} & 0 \\ 0 & \beta^{-1}
\end{pmatrix} = \begin{pmatrix}
    \alpha - \beta \delta \beta^{-1} & \mathrm{Id} \\ \beta \beta^{-1*} & 0
\end{pmatrix} := S.\]

The matrix $S$ has the form required for applying the criterion in the proof of \cite[Theorem 1.1]{jordan-larson}, meaning that $S_{11}$ and $S_{21}$ are simultaneously diagonalizable. If we denote the eigenvalues of $S_{21}$ as $\mu_1 = -\lambda, \mu_2 = -\lambda^{-1}$, then $\lambda_2 / \lambda_1 = \lambda$. Then by \cite[Claim 4.2]{jordan-larson}, the associated $\mathbb{Z}/p\mathbb{Z}$-extension is group-theoretical if and only if $\lambda \in \mathbf{F}_q$, that is, if the quotient (in any order) of the eigenvalues of $S_{11}$ lies in $\mathbf{F}_q$.
\end{proof}

\begin{theorem}\label{th: main}
The fusion category  $\operatorname{Vec}_{(\mathbb{F}_{q^2}, N)}^{p, \alpha}$ is non-group theoretical.
\end{theorem}
\begin{proof}
We will apply the criterion of Proposition \ref{Prop:criterion} to the matrix
\[\alpha_g = \begin{pmatrix}
    \alpha & \beta \\
    \gamma & \delta
\end{pmatrix} \in O(\mathbf{F}_{q^2} \oplus \mathbf{F}_{q^2}^*, Q),\]
where $g(v) = cv$. Note that in our specific situation, $\alpha + \beta \delta \beta^{-1} = \mathrm{Id} + g$. Hence, we need to ascertain whether $\frac{\mu_1}{\mu_2} \notin \mathbf{F}_q$, where $\mu_i$ are the eigenvalues of $\mathrm{Id} + g$.

As $g$ is orthogonal, its eigenvalues have the form $\beta$ and $\beta^{-1}$, where $\beta \notin \mathbf{F}_q$ since $g$ is not diagonalizable over $\mathbf{F}_q$. Note that the action of the Galois group permutes the eigenvalues of $g$, thus $\sigma(\beta) = \beta^{-1}$. Now, as the eigenvalues of $\mathrm{Id} + g$ are $1 + \beta$ and $1 + \beta^{-1}$, the criterion is based on checking whether $\lambda=\frac{1 + \beta}{1 + \beta^{-1}}$ (or equivalently $\lambda = \frac{1 + \beta^{-1}}{1 + \beta}$) is fixed by the action of the Galois group. Assume that $\sigma(\lambda) = \lambda$, meaning $\frac{1 + \beta}{1 + \beta^{-1}} = \frac{1 + \beta^{-1}}{1 + \beta}$, which is equivalent to
\begin{align}\label{eq:beta_fixed_by_galois}
\beta^4 + 2\beta^3 - 2\beta - 1 = 0.
\end{align}
However, $$x^4 + 2x^3 - 2x - 1 = (x + 1)^3(x - 1),$$ then \eqref{eq:beta_fixed_by_galois} implies that $\beta = 1$ or $\beta = -1$, which is a contradiction. The contradiction arose from assuming that $\lambda \in \mathbf{F}_q$.

Therefore, by Proposition \ref{Prop:criterion}, the category $\operatorname{Vec}_{(\mathbb{F}_{q^2}, N)}^{p, \alpha}$ is non-group-theoretical.

\end{proof}

\begin{corollary}
The modular category $\left(\operatorname{Vec}_{(\mathbb{F}_{q^2}, N)}^{p, \alpha})^{\mathbb{Z}/p\mathbb{Z}}\right)$ is non-group-theoretical.
\end{corollary}

\begin{proof}
The de-equivariantization functor defines a surjective tensor functor \[(\operatorname{Vec}_{(\mathbb{F}_{q^2}, N)}^{p, \alpha})^{\mathbb{Z}/p\mathbb{Z}} \to \operatorname{Vec}_{(\mathbb{F}_{q^2}, N)}^{p, \alpha}.\] Hence, if \((\operatorname{Vec}_{(\mathbb{F}_{q^2}, N)}^{p, \alpha})^{\mathbb{Z}/p\mathbb{Z}}\) is group-theoretical, it follows from \cite[Proposition 8.44]{ENO} that \(\operatorname{Vec}_{(\mathbb{F}_{q^2}, N)}^{p, \alpha}\) is also group-theoretical, which contradicts Theorem \ref{th: main}.
\end{proof}

\begin{theorem}
Let \(p\) and \(q\) be odd primes with \(p < q\). If there exists a non-group-theoretical modular category of dimension \(p^2q^2\), then \(p \mid (q+1)\), and the non-group-theoretical modular categories are of the form \(\left(\operatorname{Vec}_{(\mathbb{F}_{q^2}, N)}^{p, \alpha}\right)^{\mathbb{Z}/p\mathbb{Z}}\), for some 
$\alpha \in H^3(\mathbb{Z}/p\mathbb{Z},\mathbb C^*)\cong \mathbb{Z}/p\mathbb{Z}$.
\end{theorem}

\begin{proof}
Let $\cB$ a non-group-theoretical modular category of dimension $p^2q^2$. Using the same ideas of \cite[Theorem 4.2]{AIM2012} we have that the item  (c) is corrected as: $p|(q^2-1)$ and $\FPdim(\cB_{pt})=p$. 

Now $\cB_{pt}$ cannot be modular, since then $\cB=\cB_{pt}\boxtimes \cB_{ad}$ as braided fusion categories and then group-theoretical. Hence, $\cB_{pt}$ is Tannakian and we have that the $\Z/p\Z$-condensation $[\cB_{\mathbb{Z}/p\mathbb{Z}}]_e$ is a modular category of dimension $q^2$, which must be pointed, i.e. a metric group category of the form $\cC(\Z/q^2\Z,Q)$ or $\cC((\Z/q)^2,P)$ where $Q,P$ are non-degenerate quadratic forms \cite[Section 8.4]{EGNO}.

Thus $\cB$ is a $\Z/p\Z$-gauging of one of the above metric group categories. Since $O(\Z/q^2\Z,Q) = \langle \pm \id \rangle$, it admits only the trivial $\Z/p\Z$-action, resulting in a group-theoretical modular category. Consequently, $\cB$ is a $\Z/p\Z$-gauging of $\cC((\Z/q)^2,P)$. Here, we encounter two distinct quadratic forms up to equivalence: an anisotropic (or elliptic) quadratic form corresponding to $(\mathbf{F}_{q^2},N)$ and an isotropic (or hiperbolic) form of the form $t(x_1,x_2)=x_1^2+x_2^2$. The isotropic orthogonal group is  a dihedral group of order $2(q-1)$, explicitly given by $$\left\{ \begin{pmatrix}
    a &0\\ 0 & a^{-1}
\end{pmatrix} \text{(rotations)}, \begin{pmatrix}
    0 &a^{-1}\\ a & 0
\end{pmatrix} \text{(reflections)} :a\in \mathbb F_q^* \right\}.$$ Therefore, in order to have a non-trivial $\mathbb{Z}/p\mathbb{Z}$-action, $p$ must divide $q-1$. However, according to \cite[Theorem 4.8]{AIM2012} (with the additional condition $p \mid q-1$), $\mathcal{B}$ would be group-theoretical. Consequently, since $\mathcal{B}$ is non-group-theoretical, it follows that $p$ divides $q+1$. Therefore, we only need to consider the anisotropic case where the orthogonal group is a dihedral group of order \(2(q+1)\). Since every cyclic subgroup of odd order in a dihedral group is unique, there is basically a unique \(\mathbb{Z}/p\mathbb{Z}\)-action, which corresponds to the modular categories described in Definition \ref{def: modular}.

\end{proof}

\section{Group-theoretical modular categories}\label{sec: group-theoretical}

The canonical example of a group-theoretical modular category is the Drinfeld center of a pointed fusion category $\Vec_G^{\omega}$. These categories correspond to the category of representations of the twisted Drinfeld double. However, there are several examples of group-theoretical modular categories that are not Drinfeld centers. The simplest example corresponds to pointed modular categories of prime dimension. Yet, it is also possible to construct non-pointed, group-theoretical modular categories that are not Drinfeld centers. For example, $\mathbb{Z}/2\mathbb{Z}$-extensions associated with fermions of Drinfeld centers can be considered, since they change the central charge, \cite{bruillard2017fermionic}.

The next result shows that in the case of group-theoretical modular categories of dimension $p^2q^2$, we only have the pointed ones and the twisted Drinfeld centers of non-abelian groups of order $pq$, studied in detail in \cite{mignard2021modular}.

\begin{proposition}
If $\mathcal{B}$ is a group-theoretical modular category of dimension $p^2q^2$, where $p$ and $q$ are distinct primes, then $\mathcal{B}$ is either pointed or is the twisted Drinfeld double of a non-abelian group.
\end{proposition}
\begin{proof}
By \cite[Proposition 10]{CGPW}, it is known that every group-theoretical modular category $\mathcal B$ can be obtained as a gauging of a pointed modular category $\mathcal{P}$ by the trivial homomorphisms $G \to \operatorname{Pic}(\mathcal{P})$, and the dimension of $\mathcal{B}$ is $|\mathcal{P}||G|^2$.

If $\mathcal{B}$ has dimension $p^2q^2$ where $p < q$, then $(|G|, |\mathcal{P}|) = 1$ so $H^2(G, A) = 0$, and then $\mathcal{B} \cong \mathcal{P} \boxtimes \mathcal{Z}(\Vec_G^\omega)$ as braided tensor categories. Now, if $\mathcal{P} \neq \Vec$, then the order of $G$ is trivial or prime, hence $G$ is either the trivial group or a cyclic group, and $\mathcal{Z}(\Vec_G^\omega)$ is pointed. Therefore, $\mathcal{B}$ is pointed.
\end{proof}

\newcommand{\etalchar}[1]{$^{#1}$}


\begin{thebibliography}{CGPW16}

\bibitem[ABPP24]{Palcouxetc}
Max~A. Alekseyev, Winfried Bruns, Sebastien Palcoux, and Fedor~V. Petrov.
\newblock Classification of modular data of integral modular fusion categories up to rank 13, 2024.

\bibitem[BBCW19]{BBCW}
Maissam Barkeshli, Parsa Bonderson, Meng Cheng, and Zhenghan Wang.
\newblock Symmetry fractionalization, defects, and gauging of topological phases.
\newblock {\em Phys. Rev. B}, 100:115147, Sep 2019.

\bibitem[BGH{\etalchar{+}}14]{AIM2012}
Paul Bruillard, C\'{e}sar Galindo, Seung-Moon Hong, Yevgenia Kashina, Deepak Naidu, Sonia Natale, Julia~Yael Plavnik, and Eric~C. Rowell.
\newblock Classification of integral modular categories of {F}robenius-{P}erron dimension {$pq^4$} and {$p^2q^2$}.
\newblock {\em Canad. Math. Bull.}, 57(4):721--734, 2014.

\bibitem[BGH{\etalchar{+}}17]{bruillard2017fermionic}
Paul Bruillard, C{\'e}sar Galindo, Tobias Hagge, Siu-Hung Ng, Julia~Yael Plavnik, Eric~C Rowell, and Zhenghan Wang.
\newblock Fermionic modular categories and the 16-fold way.
\newblock {\em J. Math. Phys.}, 58(4), 2017.

\bibitem[BN13]{burciu2013fusion}
Sebastian Burciu and Sonia Natale.
\newblock Fusion rules of equivariantizations of fusion categories.
\newblock {\em J. Math. Phys.}, 54(1), 2013.

\bibitem[CGPW16]{CGPW}
Shawn~X. Cui, C\'{e}sar Galindo, Julia~Yael Plavnik, and Zhenghan Wang.
\newblock On gauging symmetry of modular categories.
\newblock {\em Comm. Math. Phys.}, 348(3):1043--1064, 2016.

\bibitem[DGNO10]{DGNO}
Vladimir Drinfeld, Shlomo Gelaki, Dmitri Nikshych, and Victor Ostrik.
\newblock On braided fusion categories. {I}.
\newblock {\em Selecta Math. (N.S.)}, 16(1):1--119, 2010.

\bibitem[DGP{\etalchar{+}}21]{zesting}
Colleen Delaney, C\'{e}sar Galindo, Julia Plavnik, Eric~C. Rowell, and Qing Zhang.
\newblock Braided zesting and its applications.
\newblock {\em Comm. Math. Phys.}, 386(1):1--55, 2021.

\bibitem[DN21]{davydov2021braided}
Alexei Davydov and Dmitri Nikshych.
\newblock Braided {P}icard groups and graded extensions of braided tensor categories.
\newblock {\em Selecta Math. (N.S.)}, 27(4):Paper No. 65, 87, 2021.

\bibitem[EGNO15]{EGNO}
Pavel Etingof, Shlomo Gelaki, Dmitri Nikshych, and Victor Ostrik.
\newblock {\em Tensor categories}, volume 205 of {\em Mathematical Surveys and Monographs}.
\newblock American Mathematical Society, Providence, RI, 2015.

\bibitem[ENO05]{ENO}
Pavel Etingof, Dmitri Nikshych, and Viktor Ostrik.
\newblock On fusion categories.
\newblock {\em Ann. of Math. (2)}, 162(2):581--642, 2005.

\bibitem[ENO10]{ENO3}
Pavel Etingof, Dmitri Nikshych, and Victor Ostrik.
\newblock Fusion categories and homotopy theory.
\newblock {\em Quantum Topol.}, 1(3):209--273, 2010.
\newblock With an appendix by Ehud Meir.

\bibitem[GNN09]{GNN}
Shlomo Gelaki, Deepak Naidu, and Dmitri Nikshych.
\newblock Centers of graded fusion categories.
\newblock {\em Algebra Number Theory}, 3(8):959--990, 2009.

\bibitem[JL09]{jordan-larson}
David Jordan and Eric Larson.
\newblock On the classification of certain fusion categories.
\newblock {\em J. Noncommut. Geom.}, 3(3):481--499, 2009.

\bibitem[LKW16]{LKW}
Tian {Lan}, Liang {Kong}, and Xiao-Gang {Wen}.
\newblock {Modular Extensions of Unitary Braided Fusion Categories and 2+1D Topological/SPT Orders with Symmetries}.
\newblock {\em Comm. Math. Phys.}, September 2016.

\bibitem[MS21]{mignard2021modular}
Michaël Mignard and Peter Schauenburg.
\newblock Modular categories are not determined by their modular data, 2021.

\end{thebibliography}
\end{document}